\documentclass[12pt]{amsart}

\usepackage{amssymb}
\usepackage{pb-diagram}
\usepackage{enumitem}
\usepackage[dvipsnames]{xcolor}

\usepackage{graphicx}
\usepackage{amsmath}
\numberwithin{equation}{section}
\usepackage[all]{xy}

\newtheorem{Theorem}{Theorem}[section]

\newtheorem{Lemma}[Theorem]{Lemma}
\newtheorem{Corollary}[Theorem]{Corollary}
\newtheorem{Definition}[Theorem]{Definition}
\newtheorem{Remark}[Theorem]{Remark}

\numberwithin{equation}{section}

\begin{document}

\baselineskip=16pt

\title{On the Segre Invariant for Rank Two Vector Bundles on
$\mathbb{P}^2$}

\author{L. Roa-Leguizam\'on}

\address{Instituto de F\'{\i}sica y Matem\'aticas \newline Universidad
Michoacana de San Nicol\'as de Hidalgo \newline Edificio C3,
Ciudad Universitaria \newline C.P.58040 Morelia, Mich. M\'exico.}

\email{leonardo.roa@cimat.mx}

\author{H. Torres-L\'opez}

\address{CONACyT - U. A. Matem\'aticas, U. Aut\'onoma de
Zacatecas
\newline  Calzada Solidaridad entronque Paseo a la
Bufa, \newline C.P. 98000, Zacatecas, Zac. M\'exico.}

\email{hugo@cimat.mx}

\author{A. G. Zamora}

\address{U. A. Matem\'aticas, U. Aut\'onoma de
Zacatecas
\newline  Calzada Solidaridad entronque Paseo a la
Bufa, \newline C.P. 98000, Zacatecas, Zac. M\'exico.}

\email{alexiszamora06@gmail.com}

\thanks{This paper was partially supported by CONACyT Grant CB-257079. The first author acknowledges
the financial support of Fondo Institucional de Fomento Regional
para el Desarrollo Cient\'ifico, Tecnol\'ogico y de Innovaci\'on,
FORDECYT 265667.}

\subjclass[2010]{}


\date{\today}

\begin{abstract} We extend the concept of Segre's Invariant to vector bundles on
a surface $X$. For  $X=\mathbb{P}^2$ we determine what numbers can
appear as the Segre Invariant of a rank $2$ vector bundle with
given Chern's classes. The irreducibility of strata with fixed
Segre's invariant is proved and its dimensions are computed.
Finally, we present applications to the Brill-Noether's Theory for
rank $2$ vector bundles on $\mathbb{P}^2.$ \end{abstract}

\maketitle

\section{Introduction}

We work over the  field of complex numbers $\mathbb{C}$. Given a
coherent sheaf $F$ on a variety $X$ we write $H^i(F)$ instead of
$H^i(X,F)$ and $h^i(F)$ for the corresponding dimension.

Let $C$ be a non-singular irreducible complex projective curve of
genus $g$. Let $E$ be a vector bundle of rank $n$ and degree $d$
over $C$.  For any integer $m$, $1 \leq m <  n$ the $m-$Segre
invariant is defined by
\[S_m(E) := \min\{m .\deg\, E - n .\deg \, F\}
\]
where the minimum is taken over the subbundles F of $E$ of rank m.
This invariant induces a stratification of the moduli space
$M(n,d)$ of stable vector bundles of rank $n$ and degree $d$ over
$C$.  This stratification has been studied by several authors (see
for instance \cite{Lange}, \cite{Lange-Narasimhan},
\cite{Brambila-Lange} and \cite{Russo-Teixidor}) in order to get
topological and geometric properties of $M(n,d)$.

Given a non-singular, projective surface $X$, the moduli space
$M_{X}(n,c_1,c_2)$ of rank $n$ vector bundles with Chern classes
$c_1$ and $c_2$ was constructed by Maruyama \cite{Maruyama}.
However, relatively little is known about its geometry and
subvarieties. The aim of this paper is to define the Segre
invariant for rank $2$ vector bundles on surfaces, with emphasis
in the case $X=\mathbb{P}^2$. Since this invariant defines a
semicontinuous function on the families of vector bundles of rank
$2$ on X, we get a stratification of the moduli space
$M_{\mathbb{P}^2}(2;c_1,c_2)$ into locally closed subvarieties
$M_{\mathbb{P}^2}(2;c_1,c_2;s)$ of vector bundles with fixed
Segre's invariant $s$.

Section 2 introduces the Segre invariant on surfaces and collects
a number of results that will be subsequently used. Section 3 is
the core of the paper. The main issue is to determine what numbers
can appear as the Segre invariant of a vector bundle on
$\mathbb{P}^2$ with fixed characteristic classes. The answer is
given by:

\textbf{Theorem 3.1} \hspace{.05cm} \begin{enumerate}
     \item  \textit{ Let $c_2 \geq 2$ and $k \in \mathbb{N}$.  Then a vector bundle $E\in M_{\mathbb{P}^2}(2,0,c_2)$
     with $S(E)=2k$ exists if and only
     if $k^2+k\le c_2$. Furthermore, $E$ fits in an exact
     sequence:}

\[0 \longrightarrow \mathcal{O}_{\mathbb{P}^2}(-k)  \longrightarrow E  \longrightarrow \mathcal{O}_{\mathbb{P}^2}(k)\otimes I_Z \longrightarrow 0,\]
\textit{with $Z \subset \mathbb{P}^2$ of codimension $2$, and
$\mathcal{O}_{\mathbb{P}^2}(-k) \subset{E}$  maximal.}
    \item \textit{Let $c_2 \geq 1$ and $k \in \mathbb{N}$.  Then a vector bundle $E\in M_{\mathbb{P}^2}(2,-1,c_2)$ with $S(E)=2k-1$ exists if and only
     if $k^2\le c_2$. Furthermore, $E$ fits in an exact sequence:}

\[0 \longrightarrow \mathcal{O}_{\mathbb{P}^2}(-k)  \longrightarrow E  \longrightarrow \mathcal{O}_{\mathbb{P}^2}(k-1)\otimes I_Z \longrightarrow 0,\]
\textit{with $Z \subset \mathbb{P}^2$ of codimension $2$, and
$\mathcal{O}_{\mathbb{P}^2}(-k) \subset{E}$  maximal.}
\end{enumerate} {}

The idea of the proof is to apply Serre's construction in order to
obtain the desired extension and then to show that the line
sub-bundle $\mathcal{O}_{\mathbb{P}^2}(-k)$ is, indeed, maximal.
Note that, since Segre's invariant is invariant under tensor
product with line bundles, it is sufficient to consider rank 2
vector bundles with $c_1(E)=0, -1$.

Theorem \ref{TheoremA} states in particular that the elements of
$M_{\mathbb{P}^2}(2;c_1,c_2;s)$ are parameterized by extensions of
stable vector bundles. Thus, the dimension of the stratum is
obtained by \lq\lq counting parameters" of such extensions. This
idea is formalized in Section 4, the precise statement being:

\textbf{Theorem 4.1} \label{dimensionstratum}
\begin{enumerate}
\item \textit{ Let $c_2 \geq 2$,  $k \in \mathbb{N}$ such that $k^2+k \leq c_2$. Then $M_{\mathbb{P}^2}(2;0,c_2;2k)$
is an irreducible variety of dimension:}
\[\begin{cases}
3c_2+k^2+3k-2, & \text{if $c_2 > k^2+3k+1$} \\
4c_2-3, & \text{if $c_2 \leq  k^2+3k+1$.} \\
\end{cases} \]
\item \textit{Let $c_2 \geq 1$,  $k \in \mathbb{N}$ such that $k^2 \leq c_2$.
Then $M_{\mathbb{P}^2}(2;-1,c_2;-1+2k)$ is an irreducible variety
of dimension:}
\[\begin{cases}
3c_2+k^2+2k-4, & \text{if $c_2 > k^2+2k$} \\
4c_2-4, & \text{if $c_2 \leq  k^2+2k$.} \\
\end{cases} \]
\end{enumerate}

Next, observe that the morphism:

$$0 \longrightarrow \mathcal{O}_{\mathbb{P}^2}(-k)  \longrightarrow
E,$$ leads immediately to the existence of a global section of
$E(-k)$. Thus, Segre's invariant is naturally related to the
existence of vector bundles admitting at least a section. This is
a first connection with Brill-Noether Theory. Section 5 is devoted
to investigate such connection. The main result is:

\textbf{Theorem 5.5}
\begin{enumerate}
\item \textit{Let $r, c_2,k$ be  integer  numbers satisfying  $r^2+2 \leq c_2$ and
$k<r$. Then, a vector bundle $E\in M_{\mathbb{P}^2}(2;2r,c_2,2k)$
exists such that} $h^0(E)\ge  (r-k)^2+4(r-k)+3$.

\item  \textit{Let $r, c_2$ be  integer  numbers satisfying  $r^2-r+1 \leq c_2$. Then, a vector bundle $E\in M_{\mathbb{P}^2}(2;2r-1,c_2,2k-1)$
exists such that $h^0(E)\ge t$, where}

$$ t=  \begin{cases}
(r-k)^2+4(r-k)+3  & \textit{ if }  k\neq 1, \\
r^2+r-1  & \textit{ if } k= 1. \end{cases}  $$
\end{enumerate}

We finish the paper by computing a lower bound for the dimension
of these Brill-Noether's varieties. As a consequence of our
results, the existence of nonempty Brill-Noether's locus with
negative Brill-Noether number can be deduced.

\section{Segre invariant of vector bundles on Surfaces}

We start this section by recalling the main results about vector
bundles on surfaces that shall be use in the sequel. For a further
treatment of the subject see \cite{Friedman} and
\cite{Huybrechts-Lehn}.

Let $X$ be a smooth, irreducible complex  projective surface and
let $H$ be an ample divisor on $X$. Let $\mathcal{E}$ be a torsion
free coherent sheaf on $X$ with fixed Chern classes
$c_i(\mathcal{E}) \in H^{2i}(X, \mathbb{Z})$ for $i=1,2$. The
\textit{$H-$slope} of $\mathcal{E}$ is defined as the rational
number
\[\mu_H(\mathcal{E}) := \frac{c_1(\mathcal{E}). H}{rk \, \mathcal{E} },\]
where $c_1(\mathcal{E}).H$ is the degree of $\mathcal{E}$ with
respect to $H$ and is denoted by $deg_H(\mathcal{E})$.

\begin{Definition}
\emph{ A  torsion free coherent sheaf $\mathcal{E}$  is
\textit{$H$-stable} (resp. \textit{$H$-semistable}) if for every
nonzero subsheaves $\mathcal{F}$ of smaller rank, we have
\[\mu_H(\mathcal{F})<\mu_H(\mathcal{E}) \text{ (resp. $\leq$). }\] }
\end{Definition}

Let $\mathcal{E}$ be a coherent sheaf on $X$. The dual of
$\mathcal{E}$ is the sheaf $\mathcal{E}^\vee =
\mathcal{H}om(\mathcal{E}, \mathcal{O}_X)$.  If the natural map
$\mathcal{E} \longrightarrow \mathcal{E}^{\vee\vee}$ of
$\mathcal{E}$ to its double dual is an isomorphism, we say that
$\mathcal{E}$ is reflexive.  In particular, any locally free sheaf
is reflexive. For more details related to reflexive sheaves see
\cite{Hartshorne}.

\begin{Remark}\label{stablevectorbundle}
\emph{Let $E$ be a vector bundle on $X$. Then $E$ is $H$-stable
(resp. $H$-semistable) if for all proper subbundle $F$, we have
$\mu_H(F)<\mu_H(E)$ (resp. $\leq$). Indeed, let $\mathcal{F}
\subset E$ be a proper subsheaf, hence $\mathcal{F}$ is free
torsion and  a canonical embedding exits $\mathcal{F}\rightarrow
\mathcal{F}^{\vee\vee}$ which fits  in  the following diagram:
\begin{equation*}
 \xymatrix{   \mathcal{F}  \ar[r] \ar@{^{(}->}[d]      &  E \ar@{^{(}->}[d] & \\
                 \mathcal{F}^{\vee\vee} \ar[r]_{}& E^{\vee\vee}. \\  }
\end{equation*}
Since $\mathcal{F}^{\vee\vee}$ is a reflexive sheaf and
$c_1(\mathcal{F})=c_1(\mathcal{F}^{\vee\vee})$, it follows that
$E$ is $H$-stable (resp. $H$-semistable) if for any proper
reflexive sheaf $\mathcal{F}$ we have $\mu_H(\mathcal{F}) <
\mu_H(E)$ (resp. $\leq$). Moreover, since  the singular points of
reflexive sheaf $\mathcal{F}$ has codimension greater than two,
this implies that $\mathcal{F}$ is a vector bundle. }
\end{Remark}

The following lemma allows to establish a relationship between
line subbundles of $E$ and extensions.

\begin{Lemma}\cite[Chapter 2. Proposition 5.]{Friedman} \label{factors}
Let $\phi:L \longrightarrow E$ be a sub-line bundle.  Then  a
unique effective divisor $D$ on $X$ exists,  such that the map
$\phi$ factors through the inclusion $L \longrightarrow L \otimes
\mathcal{O}_X(D)$ and such that $E/(L \otimes \mathcal{O}_X(D))$
is torsion free.
\end{Lemma}

From Lemma \ref{factors} it follows that if $L \subset E$ is
maximal , then $E$ admits an extension
\[0 \longrightarrow L \longrightarrow E \longrightarrow L'\otimes I_Z \longrightarrow 0,\]
where $L' \otimes I_Z$ is torsion free and $I_Z$ denote the ideal
sheaf of a subscheme $Z$ of codimension $2$. Note that
\begin{align*}
    c_1(E) & = c_1(L) + c_1(L'),\\
    c_2(E) & = c_1(L)\cdot c_1(L')+ l(Z),
\end{align*}
where $l(Z)$ denote the length of $Z$.

Now, we extend the concept of Segre's invariant for vector bundles
on curves to vector bundles of rank $2$ on surfaces.

\begin{Definition} \label{SegreInvariant}
\begin{em}
Let  $H$ be an ample divisor on $X$. For  a vector bundle $E$ of
rank $2$ on $X$ we define the \textit{Segre invariant} $S_H(E)$ by
\[S_H(E) := 2 \min \{\mu_H(E)-\mu_H(L)\},\]
where the minimum is taken over all line subbundles $L$ of $E$.
\end{em}
\end{Definition}

Note that Definition \ref{SegreInvariant} is equivalent to
\[S_H(E) = \min \{c_1(E).H- 2L.H\}.\]

Segre's invariant is always a finite number: in case
$X=\mathbb{P}^2$ it follows from Serre's Vanishing Theorem, in
general we have:

\begin{Lemma} Let $X$ be a projective smooth surface, $H$ an ample
line bundle on $X$ and $E$ a rank $2$ vector bundle on $X$. Then,
the set:

$$\{ L.H : L\subset E, L \text{ a line bundle} \},$$
is bounded from above.\end{Lemma}

\begin{proof} Let $L_1$, $L_2$ be such that an exact sequence:

$$0\to L_1 \to E \to L_2\otimes I_Z\to 0$$
exists, with $Z$ a zero-cycle on $X$ (see Lemma \ref{factors}
above). Consider $L\subset E$. If $L\subseteq L_1$, then $L.H\le
L_1.H$. If not, then $h^0(L_1\otimes L^{-1})=0$ and in consequence
$h^0(L_2\otimes I_Z\otimes L^{-1})\neq 0$.

From the stability of $L_2\otimes I_Z$ it follows that:

$$ L.H < c_1(L_2\otimes I_Z).H.$$
\end{proof}

The term invariant is used because $S_H(E) = S_H(E \otimes L)$ for
any line bundle $L\in Pic(X)$. By Remark \ref{stablevectorbundle}
$E$ is $H-$ stable (resp. $H-$semistable) if and only if $S_H(E) >
0$ (resp. $\geq$). We say that $L \subset E$ is \textit{maximal}
if
\[S_H(E) = c_1(E). H - 2c_1(L).H.\]

The Segre invariant induces a stratification of the moduli space
of vector bundles:

\begin{Lemma}\label{semicontinuous}
Let  $H$ be an ample divisor on $X$ and let $T$ be a variety. Let
$\mathcal{E}$ be a vector bundle of rank $2$ on $X \times T$.  The
function
\begin{eqnarray*}
S_H: T &\longrightarrow& \mathbb{Z} \\
t &\longmapsto& S_H(\mathcal{E}_t)
\end{eqnarray*}
is lower semicontinuous.
\end{Lemma}

\begin{proof}
The semicontinuity follows as a slight generalization of the
openness property of stability for which the same proof works (see
\cite[Theorem 2.8]{Maruyama})
\end{proof}

The moduli space of $H$-stable vector bundles with fixed Chern
classes $c_1$ and $c_2$ on $X$ were constructed in the 1970's by
Maruyama (see \cite{Maruyama1}).  We shall denote the moduli space
of $H-$stable vector bundles of rank $n$ and Chern classes $c_1$
and $c_2$ on $X$ by $M_{X,H}(n; c_1,c_2)$. In case $X$ is a
rational surface the dimension of moduli of rank $2$ vector
bundles  is $\dim M_{X,H}(2, c_1,c_2)=4c_2-c_1^2-3$.

\section{Segre Invariant for rank $2$ vector bundles on $\mathbb{P}^2$}

In this section we study the Segre invariant for vector bundles of
rank $2$ on $\mathbb{P}^2$, in this case a uniquely determined
integer number $k_E$ exists such that $c_1(E \otimes
\mathcal{O}_{\mathbb{P}^2}(k_E)) \in \{ -1, 0\}$. Namely
\[k_E = \begin{cases}
     -\frac{c_1(E)}{2}, & \text{if $c_1(E)$ even} \\
     -\frac{c_1(E)+1}{2}, & \text{if $c_1(E)$ odd.}
\end{cases}\]

Since $S(E) = S(E \otimes L)$ for any line bundle $L$ on
$\mathbb{P}^2$, in the remainder of this section we assume that
$E$ has degree $c_1 \in \{-1,0\}$ and second Chern class $c_2$.
Furthermore,  by stability we mean stability with respect to
$\mathcal{O}_{\mathbb{P}^2}(1)$. In fact as $Pic \, (\mathbb{P}^2)
\cong \mathbb{Z}$, there is a unique notion of stability for
$\mathbb{P}^2$.  Let $F$ be a vector bundle on $\mathbb{P}^2$,  by
abuse of notation we will use $c_1(F)$ to denote the degree of $F$
with respect to $\mathcal{O}_{\mathbb{P}^2}(1)$ and we write
$S(E)$ to denote the Segre invariant of the vector bundle $E$ of
rank $2$ on $\mathbb{P}^2$.

The following theorem is the main result of this paper. It gives
necessary and sufficient conditions for the existence of a stable
vector bundle $E$ of rank $2$, degree $c_1=0$ and $S(E)=2k$ (resp.
$c_1=-1$ and $S(E) = -1+2k$).

\begin{Theorem} \label{TheoremA}
\hspace{.05cm}
\begin{enumerate}
     \item  Let $c_2 \geq 2$ and $k \in \mathbb{N}$.  Then a vector bundle $E\in M_{\mathbb{P}^2}(2,0,c_2)$ with $S(E)=2k$ exists if and only
     if $k^2+k\le c_2$. Furthermore, $E$ fits in an exact sequence:

\[0 \longrightarrow \mathcal{O}_{\mathbb{P}^2}(-k)  \longrightarrow E  \longrightarrow \mathcal{O}_{\mathbb{P}^2}(k)\otimes I_Z \longrightarrow 0,\]
with $Z \subset \mathbb{P}^2$ of codimension $2$ and
$\mathcal{O}_{\mathbb{P}^2}(-k) \subset{E}$  maximal.
    \item Let $c_2 \geq 1$ and $k \in \mathbb{N}$.  Then a vector bundle $E\in M_{\mathbb{P}^2}(2,-1,c_2)$ with $S(E)=2k-1$ exists if and only
     if $k^2\le c_2$. Furthermore, $E$ fits in an exact sequence:

\[0 \longrightarrow \mathcal{O}_{\mathbb{P}^2}(-k)  \longrightarrow E  \longrightarrow \mathcal{O}_{\mathbb{P}^2}(k-1)\otimes I_Z \longrightarrow 0,\]
with $Z \subset \mathbb{P}^2$ of codimension $2$ and
$\mathcal{O}_{\mathbb{P}^2}(-k) \subset{E}$  maximal.
\end{enumerate}
\end{Theorem}

We need some auxiliary results:

\begin{Lemma} \cite[Lemma 1.2.5]{Okonek} \label{estableOkonek}
Let $E$ be a vector bundle on $\mathbb{P}^2$ of rank 2 and first
Chern class $c_1=0$ or $c_1 = -1$. Then $E$ is stable if and only
if $h^0(E)=0$.
\end{Lemma}

Serre's construction  provides a method for constructing rank two
vector bundles on a surface $X$ (see \cite[Chapter
5]{Huybrechts-Lehn} for more details):

\begin{Theorem}\cite[Theorem 5.1.]{Huybrechts-Lehn} \label{Cayley-Bacharach}
Let $Z \subset X$ be a local complete intersection of codimension
two in the projective non-singular surface $X$, and let $L$ and
$M$ be line bundles on $X$. Then there exists an extension
\[0 \longrightarrow L \longrightarrow E \longrightarrow M \otimes I_Z \longrightarrow 0\]
such that $E$ is locally free if and only if the pair $(L^{-1}
\otimes M \otimes \omega_X,Z)$ satisfy the Cayley-Bacharach
property:
\begin{align*}
 (CB) \,\,\,  & \text{if $Z' \subset Z$ is a sub-scheme with
 $l({\tilde Z})= l(Z)-1$ and }\\
 & \text{$s \in H^0(L^{-1} \otimes M \otimes \omega_X)$  with  $s\vert _{\tilde Z}=0$, then $s\vert_Z=0$}.
\end{align*}
\end{Theorem}

\begin{Lemma} \label{tool}
Let $Z\subset \mathbb{P}^2$ be a zero-cycle in $\mathbb{P}^2$ such
that $h^0(\mathcal{O}_{\mathbb{P}^2}(d) \otimes I_Z)=0$. Consider
a point $p\in Supp(Z)$ and write $Z={\tilde Z}+p$. Then,
\[h^0(\mathcal{O}_{\mathbb{P}^2}(d-2) \otimes I_{\tilde Z})=0.\]
In particular, $(\mathcal{O}_{\mathbb{P}^2}(d-2),Z)$ satisfies
(CB).
\end{Lemma}

\begin{proof}
Suppose  that $h^0(\mathcal{O}_{\mathbb{P}^2}(d-2) \otimes
I_{\tilde Z})\ne 0$ and let $C_1$ be one of its elements different
from zero. Consider the map
\begin{eqnarray*}
H^0(\mathcal{O}_{\mathbb{P}^2}(2)\otimes I_p)&\rightarrow & H^0(\mathcal{O}_{\mathbb{P}^2}(d) \otimes I_Z),\\
C&\mapsto& CC_1
\end{eqnarray*}
which is lineal and injective. Since
$h^0(\mathcal{O}_{\mathbb{P}^2}(2)\otimes I_p)=5$ we obtain a
contradiction.
\end{proof}

\begin{proof}\textbf{Proof of Theorem \ref{TheoremA}}
\begin{enumerate}
\item  Assume $k^2+k\le c_2$. Then  a zero cycle $Z$, locally
a complete intersection and of length $l(Z)= c_2+k^2$ exists such
that $Z$ is not contained in any curve of degree $2k-1$.  By the
Lemma \ref{tool}, the pair $(\mathcal{O}_{\mathbb{P}^2}(2k-3),Z)$
satisfies the Cayley-Bacharach property. Therefore (see Theorem
\ref{Cayley-Bacharach}), an extension
\begin{equation}\label{38imaximal}
  0 \longrightarrow \mathcal{O}_{\mathbb{P}^2}(-k) \longrightarrow E \longrightarrow \mathcal{O}_{\mathbb{P}^2}(k) \otimes I_Z \longrightarrow 0
\end{equation}
exists where $E$ is locally free and has Chern classes $c_1(E)=0$
and $c_2(E) = c_2$.

Moreover, since $Z$ is not contained in any curve of degree $2k-1$
it follows that $h^0( \mathcal{O}_{\mathbb{P}^2}(k) \otimes I_Z) =
h^0(E)=0$. Therefore, by Lemma \ref{estableOkonek} the vector
bundle $E$ is stable.

Finally, we prove that  $\mathcal{O}_{\mathbb{P}^2}(-k)$ is
maximal.  Let $\mathcal{O}_{\mathbb{P}^2}(l)$ be a line bundle
with $l<k$, from the exact sequence (\ref{38imaximal}) we have
\[0 \rightarrow \mathcal{O}_{\mathbb{P}^2}(-k+l) \rightarrow E(l) \rightarrow \mathcal{O}_{\mathbb{P}^2}(k+l) \otimes I_Z \rightarrow 0.\]

Since $l<k$ it follows that:

$$ h^0( E(l)) = h^0(\mathcal{O}_{\mathbb{P}^2}(k+l)\otimes I_Z)\le h^0(\mathcal{O}_{\mathbb{P}^2}(2k-1)\otimes I_Z)=0.$$

This implies that $\mathcal{O}_{\mathbb{P}^2}(-l)$ is not a
subbundle of $E$ and thus $\mathcal{O}_{\mathbb{P}^2}(-k)$ is
maximal.

Conversely, assume that a rank $2$ bundle with Chern classes
$c_1(E)=0$, $c_2(E)=c_2<k^2+k$ and $S(E)=2k$ exists. Therefore, we
have an exact sequence
\begin{equation}\label{maximalc2}
  0 \rightarrow \mathcal{O}_{\mathbb{P}^2}(-k) \rightarrow E \rightarrow \mathcal{O}_{\mathbb{P}^2}(k) \otimes I_Z \rightarrow 0,
\end{equation}
where $Z\subset \mathbb{P}^2$ is a local  complete intersection of
codimension $2$ with length $l(Z)= c_2+k^2$ and $
\mathcal{O}_{\mathbb{P}^2}(-k)$ is maximal. Hence,
$l(Z)<h^0(\mathcal{O}_{\mathbb{P}^2}(2k-1))$ and

\[h^0(\mathcal{O}_{\mathbb{P}^2}(2k-1)\otimes
I_Z)\neq 0.\] From the exact sequence (\ref{maximalc2}) we have:
\begin{equation*}
  0 \rightarrow \mathcal{O}_{\mathbb{P}^2}(-1) \rightarrow E(k-1) \rightarrow \mathcal{O}_{\mathbb{P}^2}(2k-1) \otimes I_Z \rightarrow
  0,
\end{equation*}
which induces the long exact sequence:
\[
    0 \rightarrow  H^0(\mathcal{O}_{\mathbb{P}^2}(-1))  \rightarrow H^0(E(k-1))
    \rightarrow H^0(\mathcal{O}_{\mathbb{P}^2}(2k-1) \otimes I_Z) \rightarrow 0 \]
Thus,
 \[h^0(E(k-1))=h^0(\mathcal{O}_{\mathbb{P}^2}(2k-1)\otimes I_Z)\neq 0,\]
but this implies that $\mathcal{O}_{\mathbb{P}^2}(-k+1)$ is a
subbundle of $E$,   contradicting  the maximality of
$\mathcal{O}_{\mathbb{P}^2}(-k)$.

\item The proof is quite analogous to the proof of  $(1)$: taking $Z\subset \mathbb{P}^2$ a local complete intersection of
codimension 2 with length $l(Z)= c_2+k^2-k$ such that $Z$ is not contained in any curve of degree $2k-2$ with $k\neq 1$ and noting  that
\[h^0(\mathcal{O}_{\mathbb{P}^2}(2k-2)) = 2k^2-k \leq l(Z) = c_2+k^2-k\]
because $k^2 \leq c_2$.  If $k=1$ the pair $(\mathcal{O}(-2),Z)$
satisfies Cayley- Bacharach and we conclude as in $(1)$.

For the converse, we can proceed analogously  to the proof of
$(1)$ by noting that $l(Z)<h^0(\mathcal{O}_{\mathbb{P}^2}(2k-2))$
and in consequence $h^0(\mathcal{O}_{\mathbb{P}^2}(2k-2)\otimes
I_Z)\neq 0$.
\end{enumerate}
\end{proof}

\begin{Corollary}
\begin{enumerate}
    \item Let $r,c_2$ be integer numbers and $k\in \mathbb{N}$ such that $c_2\geq r^2+2$.
    Then a vector bundle $E\in M_{\mathbb{P}^2}(2,2r,c_2)$ with $S(E)=2k$ exists if and only if $k^2+k+r^2\leq c_2$.
    Furthermore, $E$ fits in an exact sequence
    \begin{eqnarray*}
    0\rightarrow \mathcal{O}_{\mathbb{P}^2}(r-k)\rightarrow E \rightarrow \mathcal{O}_{\mathbb{P}^2}(r+k) \rightarrow 0.
    \end{eqnarray*}
     \item Let $r,c_2$ be integer numbers and $k\in \mathbb{N}$ such that $c_2\geq r^2+1$.
     Then a vector bundle $E\in M_{\mathbb{P}^2}(2,2r-1,c_2)$ with $S(E)=2k-1$ exists if and only if $k^2+r^2-r\leq c_2$.
     Furthermore, $E$ fits in an exact sequence
    \begin{eqnarray*}
    0\rightarrow \mathcal{O}_{\mathbb{P}^2}(r-k)\rightarrow E \rightarrow \mathcal{O}_{\mathbb{P}^2}(r+k-1) \rightarrow 0.
    \end{eqnarray*}
\end{enumerate}
\end{Corollary}

\begin{proof}
The proof follows  from Theorem \ref{TheoremA},  the fact
$S(E(-r))=S(E)$, and the formulas
\begin{eqnarray*}
c_1(E(-r))&=& c_1(E)-2r, \\
c_2(E(-r))&=&c_2(E)-rc_1(E)+r^2.
\end{eqnarray*}
\end{proof}

\section{A Stratification of the moduli space $M_{\mathbb{P}^2}(2,c_1,c_2)$}

In this section  we use the Segre invariant to induce a
stratification of the moduli space  $M_{\mathbb{P}^2}(2;c_1,c_2)$
of stable vector bundles of rank  $2$ and Chern classes $c_1$ and
$c_2$ on $\mathbb{P}^2$.  If $c_1=0$ and $c_2$ is odd (resp.
$c_1=-1$ and $c_2$ even) Le Potier (see for instance \cite[Theorem
14.6.2]{LePotier}) has shown that there exists an universal family
$\mathcal{E}$ parameterized by $M_{\mathbb{P}^2}(2;0,c_2)$ (resp.
$M_{\mathbb{P}^2}(2;-1,c_2)$). If  $c_1=0$ and $c_2$ is even
(resp. $c_1=-1$ and $c_2$ odd) working locally in the \^etale
topology we can assume that there is a family $\mathcal{E}$
parameterized by $M_{\mathbb{P}^2}(2;0,c_2)$ (resp.
$M_{\mathbb{P}^2}(2;-1,c_2)$).

Let $\mathcal{E}$ be a family of rank $2$ vector bundles on
$\mathbb{P}^2$ parameterized by $M_{\mathbb{P}^2}(2;c_1,c_2)$. By
Theorem \ref{semicontinuous}, the function $S:
M_{\mathbb{P}^2}(2;c_1,c_2) \longrightarrow \mathbb{Z}$ induces a
stratification of $M_{\mathbb{P}^2}(2;c_1,c_2)$ into locally
closed subsets
\[M_{\mathbb{P}^2}(2;c_1,c_2;s) := \{E \in M_{\mathbb{P}^2}(2;c_1,c_2) : S(E) = s\}\]
according to the value of $s$.  Without loss of generality we can
assume that if $c_1=0$ (resp. $c_1=-1$) then $s=2k$  for some $k$,
$0< k^2+k \leq c_2$ (resp. $s=-1+2k$ for some $k$, $k^2 \leq
c_2$).

\begin{Theorem} \label{dimensionstratum}
\begin{enumerate}
\item Let $c_2 \geq 2$,  $k \in \mathbb{N}$ such that $k^2+k \leq c_2$. Then $M_{\mathbb{P}^2}(2;0,c_2;2k)$
is an irreducible variety of dimension:
\[\begin{cases}
3c_2+k^2+3k-2, & \text{if $c_2 > k^2+3k+1$} \\
4c_2-3, & \text{if $c_2 \leq  k^2+3k+1$.} \\
\end{cases} \]
\item Let $c_2 \geq 1$,  $k \in \mathbb{N}$ such that $k^2 \leq c_2$.
Then $M_{\mathbb{P}^2}(2;-1,c_2;-1+2k)$ is an irreducible variety
of dimension
\[\begin{cases}
3c_2+k^2+2k-4, & \text{if $c_2 > k^2+2k$} \\
4c_2-4, & \text{if $c_2 \leq  k^2+2k$.} \\
\end{cases} \]
\end{enumerate}
\end{Theorem}

\begin{proof}
We only prove $(1)$, the proof of $(2)$ being quite analogous. Let
$c_2 \geq 2$, $k \in \mathbb{Z}$ such that $k^2+k \leq c_2$.  Let
$Hilb^{l}(\mathbb{P}^2)$ be the Hilbert scheme of zero-dimensional
subschemes of length $l= c_2 + k^2$ on $\mathbb{P}^2$ and let
$\mathcal{I}_{\mathcal{Z}_l}$ be the ideal sheaf of the universal
subscheme $\mathcal{Z}_l$ in $\mathbb{P}^2 \times
Hilb^{l}(\mathbb{P}^2)$.  Let $\mathcal{O}_{\mathbb{P}^2}(-k),
\mathcal{O}_{\mathbb{P}^2}(k)$ be line bundles on $\mathbb{P}^2$.
Let $p_1$, $p_2$ be the projections of $\mathbb{P}^2 \times
Hilb^{l}(\mathbb{P}^2)$ on $\mathbb{P}^2$ and
$Hilb^{l}(\mathbb{P}^2)$ respectively. Consider on $\mathbb{P}^2
\times Hilb^{l}(\mathbb{P}^2)$ the sheaf
$\mathcal{H}om(p_{1}^*\mathcal{O}_{\mathbb{P}^2}(k) \otimes
\mathcal{I}_{\mathcal{Z}_l},
p_{1}^*\mathcal{O}_{\mathbb{P}^2}(-k))$. Taking higher direct
image we obtain on $Hilb^{l}(\mathbb{P}^2)$ the sheaf:
\[ R^1_{p_{{2}_*}}\mathcal{H}om(p_{1}^*\mathcal{O}_{\mathbb{P}^2}(k)
\otimes \mathcal{I}_{\mathcal{Z}_l},
p_{1}^*\mathcal{O}_{\mathbb{P}^2}(-k).\]

From the semicontinuity Theorem \cite[Theorem 12.8]{Hartshorne1}
we have that the set:
\[H:=\{Z \in Hilb^l(\mathbb{P}^2) : h^0(\mathcal{O}_{\mathbb{P}^2}(2k-3) \otimes I_Z) < 1\}\]
is an open set of $Hilb^l(\mathbb{P}^2)$ which is non-empty by
Theorem \ref{TheoremA}. Restricting the sheaf
\[R^1_{p_{{2}_*}}\mathcal{H}om(p_{1}^*\mathcal{O}_{\mathbb{P}^2}(k)
\otimes \mathcal{I}_{\mathcal{Z}_l},
p_{1}^*\mathcal{O}_{\mathbb{P}^2}(-k))\] to $H$ we have that it is
locally free  because
\[H^0(\mathcal{H}om(\mathcal{O}_{\mathbb{P}^2}(k) \otimes
\mathcal{I}_{\mathcal{Z}_l},\mathcal{O}_{\mathbb{P}^2}(-k))) \cong
Hom(\mathcal{O}_{\mathbb{P}^2}(k) \otimes \mathcal{I}_Z,
\mathcal{O}_{\mathbb{P}^2}(-k))=0,\] and \[\dim \,
Ext^2(\mathcal{O}_{\mathbb{P}^2}(k)\otimes I_Z,
\mathcal{O}_{\mathbb{P}^2}(-k)) =
h^0(\mathcal{O}_{\mathbb{P}^2}(2k-3)\otimes I_Z)=0\] for any $Z
\in H$. Hence, the fiber over $Z \in H$ is
$Ext^1(\mathcal{O}_{\mathbb{P}^2}(k)\otimes I_Z,
\mathcal{O}_{\mathbb{P}^2}(-k))$.

Consider on $H$ the sheaf:
\[ \mathbb{P}\Gamma:=\mathbb{P} R^1_{p_{{2}_*}}\mathcal{H}om(p_{1}^*\mathcal{O}_{\mathbb{P}^2}(k)
\otimes \mathcal{I}_{\mathcal{Z}_l},
p_{1}^*\mathcal{O}_{\mathbb{P}^2}(-k)).\] By \cite[Lemma
3.2]{Gottsche} there exists an exact sequence:
\begin{equation} \label{univextension2}
0 \to (id\times\pi)^*p_{1}^*\mathcal{O}_{\mathbb{P}^2}(-k) \otimes
\mathcal{O}_{\mathbb{P}^2 \times \mathbb{P}\Gamma}(1) \to
\mathcal{E} \to
(id\times\pi)^*(p_{1}^*\mathcal{O}_{\mathbb{P}^2}(-k) \otimes
\mathcal{I}_{\mathcal{Z}_l}) \to 0
\end{equation}
on $\mathbb{P}^2 \times \mathbb{P}\Gamma$ such that for each $p
\in \mathbb{P}\Gamma$ the restriction $\mathcal{E}_{|_p}$ of
$\mathcal{E}$ to $\mathbb{P}^2 \times \{p\}$ is isomorphic to an
extension
\[0 \longrightarrow \mathcal{O}_{\mathbb{P}^2}(-k)  \longrightarrow E  \longrightarrow
\mathcal{O}_{\mathbb{P}^2}(k) \otimes I_Z \longrightarrow 0.\]
Define the set
\[U:= \{p \in  \mathbb{P}\Gamma  : \text{$\mathcal{E}_{|_p}$ is stable and $S(\mathcal{E}_p)=2k$}\}.\]
From Theorem \ref{TheoremA}, the lower semicontinuity of the
function $S$   and the fact that stability is an open condition we
conclude that the set $U$ is non-empty and open in $\mathbb{P}
\Gamma$. Restricting the sequence (\ref{univextension2}) to
$\mathbb{P}^2 \times U$  we have, from the universal property of
the moduli space $M_{\mathbb{P}^2}(2;0,c_2)$, a morphism
\[f_s:U \longrightarrow M_{\mathbb{P}^2}(2;0,c_2)\]
where $f_s(U)$ is precisely the stratum
$M_{\mathbb{P}^2}(2;0,c_2;2k)$. Hence,
$M_{\mathbb{P}^2}(2;0,c_2;2k)$,  being the image of an irreducible
variety under a morphism is irreducible.

We can now determine the dimension of
\[
 \dim\, M_{\mathbb{P}^2}(2;0,c_2;2k)  = \dim \, U - \dim f_s^{-1}(E) \]
for general $E$, which is equal to:
 \[
\dim\, M_{\mathbb{P}^2}(2;0,c_2;2k)  = \dim \, H + \dim \, Ext^1(
\mathcal{O}_{\mathbb{P}^2}(k)\otimes I_Z,
\mathcal{O}_{\mathbb{P}^2}(-k)) - \dim \,\mathbb{P}H^0(E(k))-1.
\]
Since $H$ is an open set of $Hilb^l(\mathbb{P}^2)$ we have:
 \begin{equation}\label{dimension}
\begin{aligned} \dim\, M_{\mathbb{P}^2}(2;0,c_2;2k)   = & \dim \,
Hilb^l(\mathbb{P}^2) +  \dim \, Ext^1(
\mathcal{O}_{\mathbb{P}^2}(k)\otimes I_Z,
\mathcal{O}_{\mathbb{P}^2}(-k)) \\ & - \dim
\mathbb{P}H^0(E(k))-1,\end{aligned}
\end{equation}

We now compute the values of \[\dim \, Ext^1(
\mathcal{O}_{\mathbb{P}^2}(k)\otimes I_Z,
\mathcal{O}_{\mathbb{P}^2}(-k)) \text{ and }\dim \,
\mathbb{P}H^0(E(k)),\] where $Z \in H$.

Note that by Serre duality
$Ext^1(\mathcal{O}_{\mathbb{P}^2}(k)\otimes I_Z,
\mathcal{O}_{\mathbb{P}^2}(-k))$ is canonically dual to \\
$Ext^1(\mathcal{O}_{\mathbb{P}^2}(-k),
\mathcal{O}_{\mathbb{P}^2}(k-3)\otimes I_Z)$.  Since
$\mathcal{O}_{\mathbb{P}^2}(-k)$ is locally free, then
\[Ext^1(\mathcal{O}_{\mathbb{P}^2}(-k),
\mathcal{O}_{\mathbb{P}^2}(k-3)\otimes I_Z) \cong
H^1(\mathcal{O}_{\mathbb{P}^2}(2k-3) \otimes I_Z ).\]

By the exact sequence
\[0 \longrightarrow \mathcal{O}_{\mathbb{P}^2}(2k-3) \otimes I_Z  \longrightarrow \mathcal{O}_{\mathbb{P}^2}(2k-3)
\longrightarrow \mathcal{O}_Z \longrightarrow 0\]
we have that
\[
  h^1(\mathcal{O}_{\mathbb{P}^2}(2k-3) \otimes I_Z)= - h^0(\mathcal{O}_{\mathbb{P}^2}(2k-3) ) + h^0(\mathcal{O}_Z),\]
because $Z \in H$ and  therefore:
 \begin{equation}\label{Ext1}
 h^1(\mathcal{O}_{\mathbb{P}^2}(2k-3) \otimes I_Z) = c_2-k^2+3k-1.
 \end{equation}

Now, we compute $h^0(E(k))$. Since $E \in
M_{\mathbb{P}^2}(2;0,c_2;2k)$ it can be written in an extension
\begin{equation}
    0 \longrightarrow \mathcal{O}_{\mathbb{P}^2}(-k) \longrightarrow E \longrightarrow \mathcal{O}_{\mathbb{P}^2}(k)\otimes I_Z \longrightarrow 0
\end{equation}
from which we get:
\[0 \longrightarrow \mathcal{O}_{\mathbb{P}^2} \longrightarrow E(k) \longrightarrow \mathcal{O}_{\mathbb{P}^2}(2k)\otimes I_Z \longrightarrow 0.\]
Therefore,
\[
h^0(E(k))  = 1 + h^0(\mathcal{O}_{\mathbb{P}^2}(2k)\otimes I_Z).
\] It follows that for $Z\in H$ general:
\begin{equation}\label{fiber}
h^0(E(k))  = \begin{cases}
1, &\text{if $c_2> k^2+3k+1$} \\
 k^2+3k-c_2+2, &\text{if $c_2 \leq k^2+3k+1.$}
 \end{cases}
\end{equation}
Replacing (\ref{Ext1}) and (\ref{fiber}) in (\ref{dimension}) we
have
\begin{align*}
 \dim\, M_{\mathbb{P}^2}(2;0,c_2;2k) = \begin{cases}
3c_2+k^2+3k-2, & \text{if $c_2 > k^2+3k+1$} \\
4c_2-3, & \text{if $c_2 \leq  k^2+3k+1$} \\
\end{cases}
\end{align*}
which proves the theorem.
\end{proof}

\begin{Corollary} \label{dimensionstratumgeneral}
\begin{enumerate}
\item Let $c_2, r\in\mathbb{Z}$ and $ k \in \mathbb{N}$ such that $r^2+k^2+k \leq c_2$. Then $M_{\mathbb{P}^2}(2;2r,c_2;2k)$
is an irreducible variety of dimension
\[\begin{cases}
3c_2-3r^2+k^2+3k-2, & \text{if }c_2 > r^2+k^2+3k+1 \\
4c_2-4r^2-3, & \text{if } c_2 \leq r^2+ k^2+3k+1. \\
\end{cases} \]
\item Let $c_2$,  $k \in \mathbb{N}$ and $r\in \mathbb{Z}$ such that $r^2-r +k^2 \leq c_2$.
Then $M_{\mathbb{P}^2}(2;2r-1,c_2;2k-1)$ is an irreducible variety
of dimension
\[\begin{cases}
3c_2+3r-3r^2+k^2+2k-4, & \text{if $c_2 > r^2 -r +k^2+2k$} \\
4c_2+4r -4r^2-4, & \text{if $c_2 \leq r^2-r +k^2+2k$.} \\
\end{cases} \]
\end{enumerate}
\end{Corollary}
\begin{proof}
For $(1)$, since $S(E)=S(E\otimes
\mathcal{O}_{\mathbb{P}^2}(-r))$, it follows that the map
\begin{eqnarray*}
M_{\mathbb{P}^2}(2;2r,c_2;2k)&\rightarrow&M_{\mathbb{P}^2}(2;0,c_2-r^2;2k)\\
E &\mapsto& E\otimes \mathcal{O}_{\mathbb{P}^2}(-r)
\end{eqnarray*}
is an isomorphism. Part (2) follows analogously.
\end{proof}

\begin{Corollary}\label{maximalstrata}
\begin{enumerate}
    \item  Let $c_2\geq r^2+2$ and $k\in\mathbb{N}$ the only integer such that $r^2+k^2+k\leq c_2\leq r^2+k^2+3k+1$.
    Then the stratum $M_{\mathbb{P}^2}(2;2r,c_2;2k)$ has the same dimension as the moduli space $M_{\mathbb{P}^2}(2;2r,c_2)$.
\item Let $c_2\geq r^2-r + 1$ and $k\in\mathbb{N}$ the only integer such that $r^2-r+k^2\leq c_2\leq r^2+ k^2+2k$.
Then the stratum $M_{\mathbb{P}^2}(2;2r-1,c_2;-1+2k)$ has the same
dimension as the moduli space $M_{\mathbb{P}^2}(2;2r-1,c_2)$
\end{enumerate}
\end{Corollary}

\begin{Remark} \begin{enumerate}
\item The uniqueness of the number $k$ satisfying the above
inequalities follows by elementary considerations. It is, indeed,
equal to the largest $k$ such that $r^2+k^2+k\leq c_2$ in (1) and
such that $r^2-r+k^2\leq c_2$ in part (2).
\item It follows by semicontinuity that the stratum with
maximal dimension is an open set in the moduli space of stable
vector bundles. \end{enumerate}
\end{Remark}

\section{Applications to Brill-Noether Theory}

In this section, we use the previous results to study the
non-emptiness of some  Brill-Noether loci in the moduli space
$M_{\mathbb{P}^2}(2;c_1,c_2)$ of stable vector bundles of  rank
$2$ and fixed Chern classes $c_1$ and $c_2$ on $\mathbb{P}^2$.

For any $t \geq 0$ the subvariety of $M_{\mathbb{P}^2}(2;c_1,c_2)$
defined as
\[W^t(2;c_1,c_2) := \{E \in M_{\mathbb{P}^2}(2;c_1,c_2) : h^0(E) \geq t\}\]
is called the \textit{$t-$Brill-Noether locus} of the moduli space
$M_{\mathbb{P}^2}(2;c_1,c_2)$ (or simply Brill-Noether locus if
there is no confusion).

The following theorem yields information about the variety
$W^t(2;c_1,c_2)$, in particular shows that $W^t(2;c_1,c_2)$ is a
determinantal variety and give a formula for the expected
dimension.  It was proved by Costa and Miro-Roig in
\cite{Costa-Miro-Roig} for every smooth projective variety of
dimension $n$.

\begin{Theorem}\cite[Corollary 2.8]{Costa-Miro-Roig} \label{determinantalvariety}
 Let $M_{\mathbb{P}^2}(2;c_1,c_2)$ be the moduli space of stable vector bundles of rank $2$ on $\mathbb{P}^2$ with fixed Chern classes $c_1$, $c_2$.
 Then, for any $t \geq 0$, there exists a determinantal variety
 \[W^t(2;c_1,c_2):= \{E \in M_{\mathbb{P}^2}(2;c_1,c_2) : h^0(E) \geq t\}.\]
 Moreover, each non-empty irreducible component of $W^t(2;c_1,c_2)$ has dimension greater or equal to the Brill-Noether number on $\mathbb{P}^2$
 \[\rho^t(2;c_1,c_2) := 4c_2-c_1^2-3 - t\left(t-\frac{c_1^2}{2}-\frac{3c_1}{2}+c_2-2\right)\]
 and
 \[W^{t+1}(2;c_1,c_2) \subset Sing(W^t(2;c_1,c_2))\]
 whenever $W^t(2;c_1,c_2) \neq M_{\mathbb{P}^2}(2;c_1,c_2)$.
\end{Theorem}

The following result allows to  establish a relationship between
the Brill-Noether locus $W^t(2;c_1,c_2)$ and the different strata
$M_{\mathbb{P}^2}(2;c_1,c_2;s)$.

\begin{Theorem}  \label{sections} Let $r,k, c_2$ be integers. Assume, \[t=\frac{(r-k+2)(r-k+1)}{2}.\] \begin{enumerate}
    \item Let $E \in M_{\mathbb{P}^2}(2;2r,c_2;2k)$. Then $ E\notin W^1(2;2r,c_2)$ if $r<k$ and  $E \in W^t(2;2r,c_2)$ if $r \geq k$.
    Moreover, the Brill-Noether number $\rho^t(2,2r,c_2)<\dim M_{\mathbb{P}^2}(2,2r,c_2;2k)$ for $c_2>>0.$

    \item Let  $E \in M_{\mathbb{P}^2}(2;2r-1,c_2;-1+2k)$. Then $ E\notin W^1(2;2r-1,c_2)$ if $r<k$ and  $E \in W^t(2;2r-1,c_2)$
    if $r \geq k$. Moreover, the Brill-Noether number $\rho^t(2,2r-1,c_2)<\dim M_{\mathbb{P}^2}(2,2r,c_2;2k-1)$ for $c_2>>0.$
\end{enumerate}
\end{Theorem}

\begin{proof}
\begin{enumerate}
    \item Since  $E \in M_{\mathbb{P}^2}(2;2r,c_2;2k)$, it follows that  $E(-r) \in M_{\mathbb{P}^2}(2;0,c_2-r^2;2k)$ because $S(E(-r))=S(E).$
    By Theorem \ref{TheoremA} there exists an extension
\begin{equation} \label{exten1}
   0 \longrightarrow \mathcal{O}_{\mathbb{P}^2}(-k) \longrightarrow E(-r) \longrightarrow \mathcal{O}_{\mathbb{P}^2}(k)\otimes I_Z \longrightarrow 0,
\end{equation}
where $Z \subset \mathbb{P}^2$ of codimension $2$ has length
$l(Z)= k^2+c_2-r^2$ and $\mathcal{O}_{\mathbb{P}^2}(-k)$ is
maximal. Note that $h^0(\mathcal{O}_{\mathbb{P}^ 2}(2k-1)\otimes
I_Z)=0$  since $\mathcal{O}_{\mathbb{P}^2}(-k)$ is maximal. From
the extension (\ref{exten1}) we get
\[0 \longrightarrow \mathcal{O}_{\mathbb{P}^2}(r-k) \longrightarrow E \longrightarrow  \mathcal{O}_{\mathbb{P}^2}(r+k) \otimes I_Z \longrightarrow 0.\]
In consequence,   $E \notin W^1(2;2r,c_2) \,\, \text{ if \,\,\,
$r<k$}$ and \[h^0(E) \geq h^0(\mathcal{O}_{\mathbb{P}^2}(r-k))= t,
\,\, \text{if $r \geq k.$} \] The Brill-Noether number is given by
\begin{eqnarray*}
\rho^t(2;2r,c_2)=4c_2 -4r^2-3-t(t+c_2-2r^2-3r-2). \end{eqnarray*}
Therefore, for $c_2>>0$ we have:
\begin{eqnarray}
\rho^t(2,2r,c_2)<3c_2-3r^2+k^2+3k-3=\dim
M_{\mathbb{P}^2}(2,2r,c_2;2k).
\end{eqnarray}
\item The proof proceeds analogously to the proof of $(1)$.
\end{enumerate}
\end{proof}

Suppose that $c_1>0$ and  the Euler-Poincar\'e characteristic
$\chi=0$, i.e. $c_2=2+\frac{c_1^2+3c_1}{2}$. It is known that the
moduli space $M_{\mathbb{P}^2}(2;c_1,c_2)$ satisfies Weak
Brill-Noether, that is, there exist  $E\in
M_{\mathbb{P}^2}(2;c_1,c_2)$ such that $h^i(E)=0$ for any $i$ (see
\cite{LePotier}, Theorem 18.1.1). Note that, by semicontinuity, if
$E$ is any sheaf with no cohomology then the cohomology also
vanishes for any general sheaf in $M_{\mathbb{P}^2}(2;c_1,c_2)$
(cf. \cite{Gottsche}, \cite{Coskun}). Using the Segre Invariant,
we can give a different proof that $M_{\mathbb{P}^2}(2;c_1,c_2)$
satisfies Weak Brill-Noether. The advantage in using Segre's
invariant lies in the fact that the open set satisfying Weak
Brill-Noether can be explicitly described as the open stratum of
$M_{\mathbb{P}^2}(2;c_1,c_2)$.

\begin{Corollary} \label{weakBN} Suppose that $c_1>0$ and
$c_2=2+\frac{c_1^2+3c_1}{2}$. Then moduli space
$M_{\mathbb{P}^2}(2,c_1,c_2)$ satisfies Weak-Brill Noether.
\end{Corollary}

\begin{proof}
Since $\chi=0$ and $c_1>0$, it follows that $h^2(E)=0$ for any
stable vector bundle with $c_1(E)=c_1$. Assume that $c_1=2r$
(resp. $c_1=2r-1$) with $r\in \mathbb{N}$. Since the
Euler-Poincar\'e characteristic $\chi=0$, it follows that
$c_2=2r^2+3r+2$ (resp. $c_2=2r^2+r+1$). Set $k=r+1,$ note that $k$
is the largest integer such that $r^2+k^2+k\leq c_2$ (resp.
$k^2+r^2-r\leq c_2$), then by Corollary \ref{maximalstrata} we
have that the stratum $M_{\mathbb{P}^2}(2;2r,c_2;2k)$ (resp.
$M_{\mathbb{P}^2}(2;2r-1,c_2;2k-1)$) is open  and by Theorem
\ref{sections} $h^0(E)=0$ for any $E\in
M_{\mathbb{P}^2}(2;2r,c_2;2k) $ (resp. $E\in
M_{\mathbb{P}^2}(2;2r-1,c_2;2k-1)$).
\end{proof}

\begin{Remark} It  is also  known that  under the conditions of
Corollary \ref{weakBN} the complement of the maximun stratum is a
reduced hypersurface (\cite{Gottsche0}, Theorem 2).\end{Remark}

For larger values of $t$ further information can be obtained. For
this we use the existence of special configurations of points:

\begin{Theorem}\label{LocusBN}
\begin{enumerate}
\item Let $r, c_2,k$ be  integer  numbers satisfying  $r^2+2 \leq c_2$ and $k<r$. Then, a vector bundle $E\in M_{\mathbb{P}^2}(2;2r,c_2,2k)$
exists such that $h^0(E)\ge  (r-k)^2+4(r-k)+3$.

\item  Let $r, c_2$ be  integer  numbers satisfying  $r^2-r+1 \leq c_2$. Then, a vector bundle $E\in M_{\mathbb{P}^2}(2;2r-1,c_2,2k-1)$
exists such that $h^0(E)\ge t$, where
\[t=\begin{cases}
(r-k)^2+4(r-k)+3  & \text{if  $k\neq 1$}, \\
r^2+r-1 & \text{if  $k= 1$.}
\end{cases} \]
\end{enumerate}
\end{Theorem}

\begin{proof}
\begin{enumerate}
    \item    Let $Z'$
    be a reduced zero-cycle of length $l(Z')=\dim \mathbb{P}H^0(\mathcal{O}_{\mathbb{P}^2}(2k-1))$
    such that $H^0(\mathcal{O}_{\mathbb{P}^2}(2k-1)\otimes I_{Z'})=\mathbb{C}. C$ for some curve $C$ of degree $2k-1$. Complete $Z'$ to
    a zero cycle $\tilde Z = Z'+ Z^{''}$ such that $Supp
    (Z^{''})\subset C$ and $l(\tilde Z)=c_2+ k^2 -r^2 -1$ and set
    $Z=\tilde Z +p$ for $p$ some point not contained in $C$.

    Then it follows from an argument similar to the one used in
    the proof of the Lemma \ref{tool} that the pair $(\mathcal{O}_{\mathbb{P}^2}(2k-3),Z)$ satisfies the
Cayley-Bacharach property (see Theorem \ref{Cayley-Bacharach}).
Indeed, if $Z=Z_0+p_0$ and $0\ne C_0 \in
H^0(\mathcal{O}_{\mathbb{P}^2}(2k-3)\otimes I_{ Z_0})$ exists,
then the map:

$$H^0(\mathcal{O}_{\mathbb{P}^2}(2)\otimes I_{p_0}) \to H^0(\mathcal{O}_{\mathbb{P}^2}(2k-1)\otimes I_{
Z}),$$

$$C \longmapsto CC_0,$$
must be injective.

Therefore an extension:
\begin{equation}\label{brillnoether1}
  0 \longrightarrow \mathcal{O}_{\mathbb{P}^2}(-k) \longrightarrow E \longrightarrow  \mathcal{O}_{\mathbb{P}^2}(k)\otimes I_Z \longrightarrow 0
\end{equation}
exists with $E\in M_{\mathbb{P}^2}(2,0,c_2-r^2)$. Since $Z-p
\subset C$ we see that any curve of degree $r-k-1$ passing through
$p$ gives rise to an element of
$H^0(\mathcal{O}_{\mathbb{P}^2}(r+k)\otimes I_Z)$. Thus,
\[
h^0(\mathcal{O}_{\mathbb{P}^2}(r+k)\otimes I_Z)\geq
h^0(\mathcal{O}_{\mathbb{P}^2}(r-k+1))-1=\frac{(r-k)^2+5(r-k)}{2}+2.
\]
In this way we obtain that $E(r)\in M_{\mathbb{P}^2}(2,2r,c_2)$
and from the exact sequence (\ref{brillnoether1}):

 \[
 h^0(E(r))=h^0(\mathcal{O}_{\mathbb{P}^2}(r-k))+h^0(\mathcal{O}_{\mathbb{P}^2}(r+k)\otimes I_Z)
 \geq (r-k)^2+4(r-k)+3,
 \] as  desired.

 \item The proof for the case $k\neq 1$ follows analogously to the proof of $(1)$. Taking $C \subset \mathbb{P}^2$
 be an irreducible curve of degree $2k-2$ and  $Z\subset \mathbb{P}^2$ be distinct points with length $l(Z)=c_2+k^2-k -r^2+r$
 such that $Z-\{p\}$ is contained in $C$ but $Z$ is not contained in $C$.  Note that the pair $(\mathcal{O}_{\mathbb{P}^2}(2k-4), Z)$
 satisfies the Cayley- Bacharach property and
 \[
h^0(\mathcal{O}_{\mathbb{P}^2}(r+k-1)\otimes I_Z)\geq
h^0(\mathcal{O}_{\mathbb{P}^2}(r-k+1))-1=\frac{(r-k)^2+5(r-k)}{2}+2.
\]
From the exact sequence
\[0 \longrightarrow \mathcal{O}_{\mathbb{P}^2}(r-k) \longrightarrow E(r) \longrightarrow  \mathcal{O}_{\mathbb{P}^2}(r+k-1)\otimes I_Z
\longrightarrow 0,\] we have:
\[
 h^0(E(r))=h^0(\mathcal{O}_{\mathbb{P}^2}(r-k))+h^0(\mathcal{O}_{\mathbb{P}^2}(r+k-1)\otimes I_Z)\geq (r-k)^2+4(r-k)+3.
 \]

We now proceed to prove the case $k=1$. Let $L\subset\mathbb{P}^2$
be a line and  let $Z\subset \mathbb{P}^2$ be distinct points with
length $l(Z)=c_2-r^2+r$ such that $Z-\{p\}$ is contained in $L$
but $Z$ is not contained in $L$. Note that the pair
$(\mathcal{O}_{\mathbb{P}^2}(-2),Z)$ satisfies the
Cayley-Bacharach property, thus, there exists an extension
\begin{equation}
  0 \longrightarrow \mathcal{O}_{\mathbb{P}^2}(-1) \longrightarrow E \longrightarrow  \mathcal{O}_{\mathbb{P}^2}\otimes I_Z \longrightarrow 0.
\end{equation}
with $c_1(E)=-1$ and $c_2(E)=c_2-r^2+r$. Moreover, $E$ is a stable
vector bundle because  $Z\neq \emptyset$.

Note that \[h^0(\mathcal{O}_{\mathbb{P}^2}(r)\otimes I_Z)\geq
h^0(\mathcal{O}_{\mathbb{P}^2}(r-1))-1=\frac{r(r+1)}{2}-1\] and
$E(r)$ is a stable vector bundle with $c_1(E(r))=2r-1$ and
$c_2(E(r))=c_2$. From the exact sequence (\ref{brillnoether1})  we
get
 \begin{equation*}
  0 \longrightarrow \mathcal{O}_{\mathbb{P}^2}(r-1) \longrightarrow E(r) \longrightarrow  \mathcal{O}_{\mathbb{P}^2}(r)\otimes I_Z \longrightarrow 0,
 \end{equation*}
 and taking cohomology we have
 \[h^0(E(r))=h^0(\mathcal{O}_{\mathbb{P}^2}(r-1))+h^0(\mathcal{O}_{\mathbb{P}^2}(r)\otimes I_Z)\geq r^2+r-1.\]
\end{enumerate}
\end{proof}

Once that Theorem \ref{LocusBN} has established the non-emptiness
of some Brill-Noether loci it is natural to search for a lower
bound for its dimensions.

\begin{Theorem} \label{dimensionlocusBN}\begin{enumerate}
     \item Let $r, k, c_2 \in \mathbb{N}$  such that  $r\geq 2$,  $3k^2-4k +r^2+2 < c_2$ and $k<r$.  Let $t= (r-k)^2+4(r-k)+3$. Then,

\[dim\, W^t(2;2r,c_2) \geq
\begin{cases}
2c_2+2k^2-2r^2+4k-2,  & \text{if $c_2> k^2+3k+r^2+1$} \\
k^2+3c_2+k-r^2-3,  & \text{if $c_2 \leq k^2+3k+r^2+1$.}
 \end{cases}\]

 \item   Let $r, k, c_2 \in \mathbb{N}$  such that  $r\geq 2$ and $3k^2-7k + r^2-r+5 < c_2$ and $k<r$.
 \[t= \begin{cases}
 (r-k)^2+4(r-k)+3, & \text{if $k \neq 1$}\\
 r^2+r-1, & \text{if $k =1$}.
 \end{cases}\]

Then,
\[dim\, W^t(2;2r-1,c_2) \geq  \begin{cases}
2c_2+2k^2-2r^2+2k+2r-4, & \text{if }c_2> k^2-2k+r^2-r \\
3c_2+k^2-3r^2+3r-4,  & \text{if } c_2 \leq k^2-2k+r^2-r,
\end{cases}\] if $k\neq 1$, and

\[ dim \, W^{t}(2;2r-1,c_2) \geq 3c_2-3r^2+3r+1, \]
if $k-1$.
\end{enumerate}
\end{Theorem}

\begin{proof}
We only prove $(1)$, the proof of $(2)$ follows by similar
arguments.

Let $l= k^2+c_2-r^2$ and let $Hilb^{\tilde{l}}(\mathbb{P}^2)$ be
the Hilbert scheme of zero-dimensional subschemes of length
$\tilde{l}= l-1$. Let $\mathcal{I}\subset
\mathbb{P}H^0(\mathcal{O}_{\mathbb{P}^2}(2k-1))$ be the subset of
irreducible curves of degree $2k-1$ and consider the variety:
\[I :=\{(p,\tilde{Z}, C) \in \mathbb{P}^2 \times Hilb^{\tilde{l}}(\mathbb{P}^2) \times \mathcal{I} :
p \notin C \, , \tilde{Z} \subset C  \}.\]

For every $C\in \mathcal{I}$ its fiber $p_3^{-1}(C)$ under the
projection onto the third factor is the irreducible variety
$(\mathbb{P}^2-C) \times S^{\tilde l}C$, which have the same
dimension for every $C$. Let $I_0\subset I$ be an irreducible
component such that $p_3: I_0 \to \mathcal{I}$ is dominant and
$\dim I_0=\dim I= k(2k+1)+l$.

Let $p_{12}: I_0\to \mathbb{P}^2 \times
Hilb^{\tilde{l}}(\mathbb{P}^2)$ be the projection onto the first
two factors. Because of the choice of $l$ and $c_2$, we have that
$\tilde l> (2k-1)^2$. Thus $p_{12}$ is injective and therefore it
is a birational morphism. In this way, an open set $
\mathcal{U}\subset p_{12}(I_0)$ exists such that $\dim I_0 =\dim
\mathcal{U}$.

Consider the finite morphism
\begin{align*}
    \phi:\mathbb{P}^2 \times Hilb^{\tilde{l}}(\mathbb{P}^2) & \longrightarrow Hilb^{l}(\mathbb{P}^2) \\
    (p, \tilde{Z}) & \longmapsto p + \tilde{Z},
\end{align*}
and the set
\begin{equation} \label{set}
    H \cap \phi(\mathcal{U}),
\end{equation}
where \[H:= \{Z \in Hilb^l(\mathbb{P}^2) : h^0(
\mathcal{O}_{\mathbb{P}^2}(2k-1) \otimes I_Z) = 0\}.\]

Note that by the proof of the Theorem \ref{TheoremA} the set
(\ref{set}) is non-empty, moreover $\phi(\mathcal{U}) \subset
\phi(p_{12}(I)) \subset H$. Therefore, $H \cap \phi(\mathcal{U})=
\phi(\mathcal{U})$

Now, we can now proceed analogously to the proof of Theorem
\ref{dimensionstratum}. Consider the sheaf
$\mathbb{P}\Gamma:=\mathbb{P}
R^1_{p_{{2}_*}}\mathcal{H}om(p_{1}^*\mathcal{O}_{\mathbb{P}^2}(k+r)
\otimes \mathcal{I}_{\mathcal{Z}_l},
p_{1}^*\mathcal{O}_{\mathbb{P}^2}(r-k))$ over $ \phi(\mathcal{U})$
and the exact sequence
\begin{equation} \label{univextension1}
0 \longrightarrow
(id\times\pi)^*p_{1}^*\mathcal{O}_{\mathbb{P}^2}(r-k) \otimes
\mathcal{O}_{\mathbb{P}^2 \times \mathbb{P}\Gamma}(1)
\longrightarrow \mathcal{E} \longrightarrow
(id\times\pi)^*(p_{1}^*\mathcal{O}_{\mathbb{P}^2}(k+r) \otimes
\mathcal{I}_{\mathcal{Z}_l}) \longrightarrow 0.
\end{equation}
on $\mathbb{P}^2 \times \mathbb{P}\Gamma$. Define the set
\[U:= \{p \in  \mathbb{P}\Gamma  : \text{$\mathcal{E}_{|_p}$ is stable and $S(\mathcal{E}_p)=2k$}\}.\]

From Theorem \ref{TheoremA}, the lower semicontinuity of the
function $S$   and stability being an open condition we conclude
that the set $U$ is non-empty and open in $\mathbb{P} \Gamma$.
Restricting the sequence (\ref{univextension1}) on $\mathbb{P}^2
\times U$ from the universal property of the moduli space
$M_{\mathbb{P}^2}(2;c_1,c_2)$ we have a morphism
\[f_t:U \longrightarrow M_{\mathbb{P}^2}(2;2r,c_2).\]

Note that, by the proof of Theorem \ref{LocusBN}  the set $f_t(U)$
is contained in the locus of Brill-Noether $W^{t}(2;2r,c_2)$ where
$t=(r-k)^2+4(r-k)+3$. Moreover, \[f_t(U) \subseteq
\overline{f_t(U)} \subseteq  W^{t}(2;2r,c_2).\]

We proceed now to determine the dimension of $Im
\overline{f_t(U)}$ and a lower bound of the locus
$W^{t}(2;2r,c_2)$,

    \[ \dim \, W^{t}(2;2r,c_2) \geq
 \dim\, Im \overline{f_t(U)}  \geq \dim \, U - dim f_t^{-1}(E) \]
which is equivalent to
 \begin{align*}
     \dim\, Im \overline{f_t(U)} & \geq
     \dim\,   \phi{(\overline{p_{12}(I)})} + \dim \, Ext^1( \mathcal{O}_{\mathbb{P}^2}(k+r)\otimes I_Z, \mathcal{O}_{\mathbb{P}^2}(r-k)) - \dim \,\mathbb{P}H^0(E(k-r))-1 \\
     & = \dim \, \phi(p_{12}(I)) +  \dim \, Ext^1( \mathcal{O}_{\mathbb{P}^2}(k+r)\otimes I_Z, \mathcal{O}_{\mathbb{P}^2}(r-k)) -
     \dim \,\mathbb{P}H^0(E(k-r))-1 \\
     & \geq \dim \, \phi(\mathcal{U}) +  \dim \, Ext^1( \mathcal{O}_{\mathbb{P}^2}(k+r)\otimes I_Z, \mathcal{O}_{\mathbb{P}^2}(r-k)) -
     \dim \,\mathbb{P}H^0(E(k-r))-1 .
 \end{align*}

From the above and the proof of Theorem \ref{dimensionstratum}  we
conclude that

\[ \dim\, Im \overline{f_t(U)} \geq
\begin{cases}
c_2+k^2-r^2+4k-2,  & \text{if $c_2>k^2+r^2+3k+1$} \\
2c_2+k-2r^2-3,  & \text{if $c_2 \leq k^2+3k+r^2+1$.}
 \end{cases}\]
as claimed.

\end{proof}


\begin{thebibliography}{99}


\bibitem {Barth}
  {\sc W.P. \ Barth, K. \ Hulek, C.A.M \ Peters,  A. \ Van de Ven.} ---
  {\it Compact complex surfaces}. Second edition. Ergebnisse der Mathematik und ihrer Grenzgebiete.
  Springer-Verlag, Berlin, 2004. xii+436 pp.



\bibitem {Brambila-Lange}
  {\sc L.\ Brambila-Paz, H. \ Lange.} ---
  {\it A stratification of the moduli space of vector bundles on curves.}
 Journal f\"ur die reine und angewandte Mathematik, 494, 173-187, (1998).

\bibitem{Costa-Miro-Roig}
 {\sc L.\ Costa,  R.M. \ Miro-Roig.}---
    {\it Brill-Noether theory for moduli spaces of sheaves on algebraic varieties}. Forum Math. 22 (2010), no. 3, 411-432.

\bibitem{Coskun}
{ \sc I. \  Coskun, J. \ Huizenga. }--- Weak {\it Brill-Noether
for rational surfaces. Local and global methods in algebraic
geometry},  Contemp. Math., 712, Amer. Math. Soc., Providence, RI,
2018.


\bibitem {Friedman}
  {\sc R.\ Friedman.} ---
  {\it Algebraic surfaces and holomorphic vector bundles}. Universitext. Springer-Verlag, New York, 1998. x+328 pp.


\bibitem{Gottsche0}
   {\sc L. \ Gottsche and A. \  Hirschowitz}.--- {\it  Weak Brill-Noether for vector bundles on the projective plane}.
   Algebraic geometry (Catania, 1993/Barcelona, 1994), Lecture Notes in Pure and Appl. Math., 200, Dekker, New York, 1998.


\bibitem {Gottsche}
  {\sc L.\ Gottsche.} ---
  {\it Change of polarization and Hodge numbers of moduli spaces of torsion free sheaves on surfaces.}
 Mathematische Zeitschrift, 223, 247-260, (1996).


\bibitem{Lange} {\sc H. \ Lange}.--- {\it Zur Klassifikation von
Regelmannigfaltigkeiten.}  Mathematische Annalen. 262 (4),
447–459, (1983).


\bibitem {Lange-Narasimhan}
  {\sc H.\ Lange, N. \ Narasimhan.} ---
  {\it Maximal subbundles of Rank two vector bundles on curves.}
 Mathematishe Annalen, 266, 55-72, (1983).

 \bibitem {LePotier}
  {\sc J.\ Le Potier. } ---
  {\it Lectures on vector bundles}. Translated by A. Maciocia. Cambridge Studies in Advanced Mathematics, 54. Cambridge University Press, Cambridge, 1997. viii+251 pp.

 \bibitem {Hartshorne}
  {\sc R.\ Hartshorne. } ---
  {\it Stable reflexive sheaves.} Math. Ann. 254 (1980), no. 2, 121-176.

   \bibitem {Hartshorne1}
  {\sc R.\ Hartshorne. } ---
   {\it Algebraic geometry}. Graduate Texts in Mathematics, No. 52. Springer-Verlag, New York-Heidelberg, 1977. xvi+496 pp.

\bibitem {Huybrechts-Lehn}
  {\sc D.\ Huybrechts, M. \ Lehn.} ---
   {\it The geometry of moduli spaces of sheaves.} Aspects of Mathematics, E31. Friedr. Vieweg  Sohn, Braunschweig, 1997. xiv+269 pp.

\bibitem{Maruyama}
  {\sc  M.\ Maruyama.} ---
{\it Openness of a family of torsion free sheaves}. J. Math. Kyoto
Univ. 16 (1976), no. 3, 627--637.

\bibitem{Maruyama1}
  {\sc M.\ Maruyama.} ---
  {\it Stable vector bundles on an algebraic surface.}
 Nagoya Math. J. Vol. 58 (1975), 25-68.

\bibitem {Okonek}
  {\sc  Okonek, Christian; Schneider, Michael; Spindler, Heinz.} ---
  .{\it Vector bundles on complex projective spaces}. Progress in Mathematics, 3. Birkhäuser, Boston, Mass., 1980. vii+389 pp.

 \bibitem {Russo-Teixidor}
  {\sc  B. \ Russo, M.\ Teixidor.} ---
  {\it On a conjecture of Lange.}
 Journal of Algebraic Geometry, 8, 483-496, (1999).

\end{thebibliography}
\end{document}